\documentclass[10pt,a4paper]{amsart}
\usepackage[margin=3.3cm]{geometry}
\usepackage{amssymb}
\usepackage{graphicx} 
\usepackage[mathscr]{euscript}
\usepackage{enumerate}
\usepackage{xspace}
\usepackage{color}
\usepackage{enumerate}
\usepackage{enumitem}
\usepackage{amsthm}
\usepackage{dsfont}
\usepackage{bbm}
\usepackage[colorlinks=true, linkcolor = blue, citecolor = blue]{hyperref}
\usepackage[numbers]{natbib}
\usepackage[backgroundcolor=white,bordercolor=red]{todonotes}
\DeclareMathAlphabet{\mathpzc}{OT1}{pzc}{m}{it}

\usepackage[nameinlink]{cleveref}

\begin{document}

\theoremstyle{plain}

\newtheorem{theorem}{Theorem}[section]
\newtheorem{lemma}[theorem]{Lemma}
\newtheorem{proposition}[theorem]{Proposition}
\newtheorem{corollary}[theorem]{Corollary}
\newtheorem{definition}[theorem]{Definition}
\newtheorem{Ass}[theorem]{Assumption}
\theoremstyle{definition}
\newtheorem{remark}[theorem]{Remark}
\newtheorem{SA}[theorem]{Standing Assumption}
\renewenvironment{proof}{{\parindent 0pt \it{ Proof:}}}{\mbox{}\hfill\mbox{$\Box\hspace{-0.5mm}$}\vskip 16pt}

\newcommand{\Y}{\mathsf{Y}}
\newcommand{\bC}{\mathbf{C}}
\newcommand{\pK}{\p^{\mathcal{K}}}
\newcommand{\loc}{{\mathrm{loc}}}
\newcommand{\of}{[\hspace{-0.06cm}[}
\newcommand{\gs}{]\hspace{-0.06cm}]}
\newcommand{\A}{\mathbf{A}}
\newcommand{\B}{\mathbb{B}}
\newcommand{\p}{\mathds{P}}
\newcommand{\Q}{\mathds{Q}}
\newcommand{\q}{\mathds{Q}}
\let\SETMINUS\setminus
\renewcommand{\setminus}{\backslash}
\renewcommand{\H}{{U}} 
\newcommand{\h}{U} 
\newcommand{\s}{\mathfrak{s}}
\def\stackrelboth#1#2#3{\mathrel{\mathop{#2}\limits^{#1}_{#3}}}
\crefname{lemma}{lemma}{lemmas}
\Crefname{lemma}{Lemma}{Lemmata}
\crefname{corollary}{corollary}{corollaries}
\Crefname{corollary}{Corollary}{Corollaries}
\crefname{listing}{assumption}{assumptions}
\Crefname{listing}{Condition}{Conditions}

\newcommand\myrot[1]{\mathrel{\rotatebox[origin=c]{#1}{$\Longrightarrow$}}}
\newcommand\NEarrow{\myrot{45}}
\newcommand\SEarrow{\myrot{-45}}

\renewcommand{\theequation}{\thesection.\arabic{equation}}
\numberwithin{equation}{section}

\newcommand\llambda{{\mathchoice
		{\lambda\mkern-4.5mu{\raisebox{.4ex}{\scriptsize$\backslash$}}}
		{\lambda\mkern-4.83mu{\raisebox{.4ex}{\scriptsize$\backslash$}}}
		{\lambda\mkern-4.5mu{\raisebox{.2ex}{\footnotesize$\scriptscriptstyle\backslash$}}}
		{\lambda\mkern-5.0mu{\raisebox{.2ex}{\tiny$\scriptscriptstyle\backslash$}}}}}

\newcommand{\tr}{\operatorname{tr}}
\newcommand{\D}{{\mathbb{D}}}
\newcommand{\E}{{\mathds{E}}}
\newcommand{\N}{{\mathbb{N}}}
\renewcommand{\P}{{\mathbf{P}}}
\newcommand{\cadlag}{c\`adl\`ag }
\newcommand{\C}{\mathbb{C}}
\newcommand{\on}{\operatorname}

\newcommand{\F}{\mathbf{F}}
\newcommand{\1}{\mathds{1}}
\newcommand{\f}{\mathcal{F}^{\hspace{0.03cm}0}}
\newcommand{\G}{\mathbf{G}}
\newcommand{\M}{\mathcal{M}^{\textup{sp}}}
\newcommand{\K}{\mathbb{K}}
\def\EM{\ensuremath{(\mathbb{EM})}\xspace}
\newcommand{\la}{\langle}
\newcommand{\ra}{\rangle}

\newcommand{\lle}{\langle\hspace{-0.085cm}\langle}
\newcommand{\rre}{\rangle\hspace{-0.085cm}\rangle}
\newcommand{\blle}{\Big\langle\hspace{-0.155cm}\Big\langle}
\newcommand{\brre}{\Big\rangle\hspace{-0.155cm}\Big\rangle}

\renewcommand{\epsilon}{\varepsilon}
\newcommand{\X}{\mathsf{X}}
\renewcommand{\F}{\mathcal{F}}
\renewcommand{\u}{\sigma} 
\renewcommand{\v}{\mu} 
\newcommand{\w}{\mathfrak{w}}
\renewcommand{\q}{\mathfrak{q}}
\renewcommand{\p}{\mathsf{P}}

\renewcommand{\P}{\mathds{P}}
\renewcommand{\A}{\mathcal{A}}
\newcommand{\Ab}{\mathbf{A}}
\newcommand{\Fb}{\mathbf{F}}
\newcommand{\bcdot}{\boldsymbol{\cdot}}
\renewcommand{\f}{\mathfrak{f}}

\title[On a Theorem of A.S. Cherny for SPDEs]{On a Theorem by A.S. Cherny for Semilinear Stochastic Partial Differential Equations} 
\author[D. Criens]{David Criens}
\author[M. Ritter]{Moritz Ritter}
\address{University of Freiburg, Ernst-Zermelo-Str. 1, 79104 Freiburg, Germany}
\email{david.criens@stochastik.uni-freiburg.de}
\email{moritz.ritter@stochastik.uni-freiburg.de}

\keywords{stochastic partial differential equation, martingale problem, weak solution, mild solution, dual Yamada--Watanabe theorem, weak uniqueness, joint weak uniqueness, pathwise uniqueness, strong uniqueness\vspace{1ex}}

\subjclass[2010]{60H15, 60G44, 60H05}

\thanks{DC acknowledges financial support from the DFG project No. SCHM 2160/15-1. \\}

\date{\today}
\maketitle

\frenchspacing
\pagestyle{myheadings}

\begin{abstract}
We consider analytically weak solutions to semilinear stochastic partial differential equations with non-anticipating coefficients driven by cylindrical Brownian motion. The solutions are allowed to take values in general separable Banach spaces. We show that weak uniqueness is equivalent to weak joint uniqueness, and thereby generalize a theorem by A.S. Cherny to an infinite dimensional setting. Our proof for the technical key step is different from Cherny's and uses cylindrical martingale problems. As an application, we deduce a dual version of the Yamada--Watanabe theorem, i.e. we show that strong existence and weak uniqueness imply weak existence and strong uniqueness.
\end{abstract}

\section{Introduction}
The classical Yamada--Watanabe theorem \cite{YW} for finite dimensional Brownian stochastic differential equations (SDEs) states that weak existence and strong (i.e. pathwise) uniqueness implies strong existence and weak uniqueness (i.e. uniqueness in law). Jacod \cite{J80} lifted this result to SDEs driven by semimartingales and extended it by showing that weak existence and strong uniqueness is equivalent to strong existence and weak joint uniqueness, i.e. uniqueness of the joint law of the solution process and its random driver. 

In view of Jacod's theorem, it is an interesting and natural question whether the converse direction in the classical Yamada--Watanabe theorem holds, i.e. whether strong existence and weak uniqueness implies weak existence and strong uniqueness. This implication is nowadays often called the \emph{dual Yamada--Watanabe theorem}. 
For finite dimensional Brownian SDEs, A.S. Cherny \cite{doi:10.1137/S0040585X97979093} answered this question affirmatively by proving that weak uniqueness is equivalent to weak joint uniqueness. 

More recently, Cherny's results have been generalized to several infinite dimensional frameworks. In \cite{ondrejat-diss,tappe20} the theorems were established for mild solutions to semilinear stochastic partial differential equations (SPDEs) and in \cite{doi:10.1142/S0219493710002991,rem20} for the variational framework.

In this short article we prove Cherny's result for analytically weak solutions to the Banach space valued semilinear SPDE  
\begin{align}\label{eq: SPDE}
  d X_t = (AX_t + \v_t (X))dt + \u_t(X) d W_t, \quad X_0 = x_0,
  \end{align} 
  where \(A\) is a densely defined operator and \(\v\) and \(\u\) are progressively measurable processes on the path space of continuous functions.
Furthermore, we deduce the dual theorem for our framework.

To the best of our knowledge, these results are new and extend previous ones in several directions. 
For instance, we study Banach space valued equations, while in \cite{tappe20} only Hilbert space valued equations are considered,
and allow non-anticipating coefficients, which are not covered in \cite{ondrejat-diss}.
In particular, as we work with analytically weak solutions instead of mild solutions, we require no geometric assumptions on the underlying Banach space and only minimal assumptions on the linearity \(A\). 
Besides providing a new result, another motivation for this article is to present a proof whose main technical step appears to us more straightforward and less technical than the classical one.
Let us explain our impression in more detail. The basic strategy, which is borrowed from the finite dimensional case and also used in \cite{ondrejat-diss,doi:10.1142/S0219493710002991,rem20},\footnote{The proof in \cite{tappe20} is indirect in the sense that it uses the method of the moving frame to transfer results from \cite{rem20} for infinite dimensional SDEs to SPDEs.} is to construct an infinite dimensional Brownian motion \(V\), independent of \(X\), such that the noise \(W\) can be recovered from the solution process \(X\) and \(V\). The technical challenge in this argument is the proof for the independence of \(X\) and \(V\).
Cherny's original proof for this used additional randomness, an enlargement of filtration and a conditioning argument. In \cite{ondrejat-diss,doi:10.1142/S0219493710002991,rem20} these ideas have been adapted to the respective infinite dimensional frameworks.
Our approach is different: We transfer ideas from \cite{criens20} for one dimensional SDEs with jumps to our continuous infinite dimensional setting and establish the independence with arguments based on cylindrical martingale problems.
More precisely, we provide martingale characterizations for weak solutions to SPDEs and infinite dimensional Brownian motion, then show that the quadratic variations of the corresponding test martingales vanish and finally deduce the desired independence with help of changes of measure. 
In comparison with Cherny's method, we work directly with \(X\) and \(V\) without introducing more randomness.
Furthermore, once the martingale characterizations are established, the arguments are quite elementary.

The paper is structured as follows: In \Cref{sec: main} we introduce our setting and state our main results: \Cref{theo: main1} and \Cref{coro: DYW}. 
At the end of \Cref{sec: main} we shortly comment on possible applications of our results.
The proof of \Cref{theo: main1} is given in \Cref{sec: pf}. To make the article as self-contained as possible, we added \Cref{app: TT}, where we collect some technical facts needed in our proofs.

Let us end the introduction with a short comment on notation and terminology: We mainly follow the standard references \cite{DePrato,roeckner15}. A detailed construction and standard properties of the stochastic integral can also be found in \cite{ondrejat-diss}.

\section{The Setting and Main Results}\label{sec: main}

Let \(\H\) be a real separable Banach space and let \(H\) be a real separable Hilbert space. We denote the corresponding norms by \(\|\cdot\|_\H\) and \(\|\cdot\|_H\) and the scalar product of \(H\) by \(\langle \cdot, \cdot\rangle_H\). As usual, the topological dual of \(H\) is identified with itself via the Riesz representation. The topological dual of \(\H\) is denoted by \(\h^*\) and we write
 \[
\langle y, y^* \rangle_\H \triangleq y^* (y), \quad (y, y^*) \in \H \times \h^*.
 \]
   The space of bounded linear operators \(H \to \H\) is denoted by \(L \triangleq L (H, \H)\) and the corresponding operator norm is denoted by \(\| \cdot \|_L\).
We define \(\C \triangleq C(\mathbb{R}_+, \H)\) to be the space of continuous functions \(\mathbb{R}_+ \to \H\). Let \(\X = (\X_t)_{t \geq 0}\) be the coordinate process on \(\C\), i.e. \(\X(\omega) = \omega\) for \(\omega \in \C\), and set \(\mathcal{C} \triangleq \sigma (\X_t, t \in\mathbb{R}_+)\) and \(\bC \triangleq (\mathcal{C}_t)_{t \geq 0}\), where \(\mathcal{C}_t \triangleq\bigcap_{s > t} \sigma (\X_u, u \in [0, s])\) for \(t \in \mathbb{R}_+\). 

Let us shortly comment on the driving noise of the SPDEs under consideration and on stochastic integration.
We call a family \(W \triangleq (\beta^k)_{k \in \mathbb{N}}\) of independent one dimensional standard Brownian motions a \emph{standard \(\mathbb{R}^\infty\)-Brownian motion}. 
It is well-known (see, e.g. \cite[Chapter~2]{roeckner15}) that any standard \(\mathbb{R}^\infty\)-Brownian motion can be seen as a trace class Brownian motion in another Hilbert space:
Let \(J\) be a one-to-one Hilbert--Schmidt embedding of \(H\) into another separable Hilbert space \((\overline{H}, \|\cdot\|_{\overline{H}}, \langle \cdot, \cdot\rangle_{\overline{H}})\) and let \((e_k)_{k \in \mathbb{N}}\) be an orthonormal basis of \(H\).
The formula 
\[
\overline{W} \triangleq \sum_{k = 1}^\infty \beta^k J e_k
\]
defines a trace class \(\overline{H}\)-valued Brownian motion with covariance \(J J^*\). 
Conversely, any trace class \(\overline{H}\)-valued Brownian motion with covariance \(J J^*\) has such a series representation.
Let \(\sigma = (\sigma_t)_{t \geq 0}\) be an \(H\)-valued progressively measurable process such that a.s.
\begin{align} \label{eq: sigma inte}
\int_0^t \|\sigma_s\|^2_H ds < \infty, \quad t \in \mathbb{R}_+.
\end{align}
Then, \(\widetilde{\sigma} \triangleq \langle \sigma, \hspace{0.01cm} \cdot\hspace{0.03cm}\rangle_H\) defines a progressively measurable process with values in \(L_2 (H, \mathbb{R})\), the space of Hilbert--Schmidt operators \(H \to \mathbb{R}\), and \(\|\widetilde{\sigma}\|_{L_2(H, \mathbb{R})} = \|\sigma\|_H\).
The stochastic integral of \(\widetilde{\sigma}\) w.r.t. a standard \(\mathbb{R}^\infty\)-Brownian motion \(W\) is defined by
\[
\int_0^\cdot \langle \sigma_s, d W_s\rangle_H \equiv \int_0^\cdot \widetilde{\sigma}_s d W_s \triangleq \int_0^\cdot \widetilde{\sigma}_s J^{-1} d \overline{W}_s \equiv \int_0^\cdot \langle \sigma_s, J^{-1} d \overline{W}_s \rangle_H, 
\]
where the stochastic integrals on the r.h.s. are defined in the classical manner (see, e.g. \cite[Chapter~2]{roeckner15}). 
We stress that this definition of the stochastic integral is independent of the choice of \(\overline{H}\) and \(J\). 
It is also well-known (see, e.g. \cite[Section 4.1.2]{DePrato}) that a standard \(\mathbb{R}^\infty\)-Brownian motion can be seen as a cylindrical Brownian motion \(\{B (x) \colon x \in H\}\) defined by the formula
\[
B (x) \triangleq \sum_{k = 1}^\infty \langle x, e_k\rangle_H \beta^k, \quad x \in H.
\]
Conversely, any cylindrical Brownian motion has such a series representation. For a simple \(H\)-valued process \(\sigma = \sum_{k = 1}^m f^k x^k\), where \(f^k\) are bounded real-valued progressively measurable processes and \(x^k \in H\), the stochastic integral of \(\sigma\) w.r.t. \(B\) can be defined by
\[
\int_0^\cdot \langle \sigma_s, d B_s \rangle_H \triangleq \sum_{k = 1}^m \int_0^\cdot f^k_s d B_s (x^k), 
\]
where the stochastic integrals on the r.h.s. are classical stochastic integrals w.r.t. one dimensional continuous local martingales.
This definition extends to more general integrands by approximation, see \cite{miro98} or \cite{ondrejat-diss}.
In particular, for any \(H\)-valued progressively measurable process \(\sigma = (\sigma_t)_{t \geq 0}\) satisfying \eqref{eq: sigma inte} it holds that
\[
\int_0^\cdot \langle \sigma_s, d W_s\rangle_H = \int_0^\cdot \langle \sigma_s, J^{-1} d \overline{W}_s\rangle_H = \int_0^\cdot \langle \sigma_s, d B_s\rangle_H.
\]

In the following we fix \(\overline{H}\) and \(J\) and identify the law of \(W\) with the law of \(\overline{W}\) seen as a probability measure on the canonical space of continuous functions \(\mathbb{R}_+ \to \overline{H}\) equipped with the \(\sigma\)-field generated by the corresponding coordinate process (which is its Borel \(\sigma\)-field when endowed with the local uniform topology). 

The input data for the SPDE \eqref{eq: SPDE} is the following:
\begin{enumerate}
\item[-]
Two processes defined on the filtered space \((\C, \mathcal{C}, \bC)\): An \(\H\)-valued progressively measurable process \(\v =(\v_t)_{t  \geq 0}\) and an \(L\)-valued progressively measurable process \(\u = (\u_t)_{t \geq 0}\), i.e. \(\u h\) is progressively measurable for every \(h \in H\).
\item[-] A densely defined operator \(A \colon D(A) \subseteq \H \to \H\) with adjoint \(A^* \colon D(A^*) \subseteq \H^* \to \H^*\) whose domain \(D(A^*)\) is sequentially weak\(^*\) dense in \(\H^*\).
\item[-] An initial value \(x_0 \in \H\). 
\end{enumerate}

\begin{remark}
Often enough \(\H\) is itself a Hilbert space, or a least a reflexive Banach space, and \(A\) is the generator of a \(C_0\)-semigroup on \(\H\). In these cases \(A\) and \(A^*\) are densely defined and in particular \(D(A^*)\) is sequentially weak\(^*\) dense.
\end{remark}

In the following definition we introduce \emph{analytically and probabilistically weak  solutions to the SPDE \eqref{eq: SPDE}} and two \emph{weak uniqueness concepts}.
\begin{definition}\label{def: weak sol}
\begin{enumerate}
\item[\textup{(i)}] We call \((\B, W)\) a \emph{driving system}, if \(\B = (\Omega, \mathcal{F}, (\mathcal{F}_t)_{t \geq 0}, \P)\) is a filtered probability space with right-continuous and complete filtration which supports a standard \(\mathbb{R}^\infty\)-Brownian motion \(W\).
\item[\textup{(ii)}]
We call \((\B, W, X)\) a \emph{weak solution} to the SPDE \eqref{eq: SPDE},
if \((\B, W)\) is a driving system and \(X\) is a continuous \(\H\)-valued adapted process on \(\B\) such that a.s.
\begin{align}\label{eq: sol inte}
\int_0^t \|\v_s(X)\|_\H ds + \int_0^t \|\u_s(X)\|_{L}^2 ds < \infty, \quad t \in \mathbb{R}_+,
\end{align}
and for all \(y^* \in D(A^*)\) a.s.
\begin{equation}\label{eq: main eq SPDE}
\begin{split}
\langle X, y^*\rangle_\H = \langle x_0&, y^*\rangle_\H +  \int_0^\cdot  \langle X_s, A^*y^*\rangle_\H ds \\&+ \int_0^\cdot \langle \v_s(X), y^*\rangle_\H ds + \int_0^\cdot \langle \u_s (X)^* y^*, d W_s \rangle_{H}.
\end{split}
\end{equation}
The process \(X\) is called a \emph{solution process} on the driving system \((\B, W)\).
\item[\textup{(iii)}] We say that \emph{weak (joint) uniqueness} holds for the SPDE \eqref{eq: SPDE}, if for any two weak solutions \((\B^1, W^1, X^1)\) and \((\B^2, W^2, X^2)\) the laws of \(X^1\) and \(X^2\) (the laws of \((X^1, W^1)\) and \((X^2, W^2)\)) coincide. The law of a solution process is called a \emph{solution measure}.
\end{enumerate}
\end{definition}

Our main result is the following.
\begin{theorem}\label{theo: main1}
Weak uniqueness holds if and only if weak joint uniqueness holds.
\end{theorem}
The proof of this theorem is given in \Cref{sec: pf} below.
We also provide a dual Yamada--Watanabe theorem for our framework. To formulate it we need more terminology.
\begin{definition}
	\begin{enumerate}
		\item[\textup{(i)}]
		We say that \emph{strong existence} holds for the SPDE \eqref{eq: SPDE}, if there exists a weak solution \((\B, W, X)\) such that \(X\) is adapted to the completion of the natural filtration of \(\overline{W}\).
	\item[\textup{(ii)}] We say that \emph{strong uniqueness} holds for the SPDE \eqref{eq: SPDE}, if any two solution processes on the same driving system are indistinguishable. 
		\end{enumerate}
\end{definition}
The classical Yamada--Watanabe theorem for the Markovian version of our framework is given by \cite[Theorem 5.3]{kunze13}.
\begin{corollary}[Dual Yamada--Watanabe Theorem] \label{coro: DYW}
	Weak Uniqueness and strong existence imply strong uniqueness and weak existence.
\end{corollary}
\begin{proof}
	Due to \Cref{theo: main1}, it suffices to show that weak joint uniqueness and strong existence imply strong uniqueness. To prove this, we follow the proof of \cite[Theorem~8.3]{J80}.
	Let \(\p\) be the unique joint law of a solution process and its driver, and let \(\mathds{W}\) be the unique law of a trace class \(\overline{H}\)-valued Brownian motion with covariance \(J J^*\). As strong existence holds, \cite[Lemmata~1.13,~1.25]{Kallenberg} imply the existence of a measurable map \(F \colon C(\mathbb{R}_+, \overline{H}) \to \C = C(\mathbb{R}_+, \H)\) such that 
	\[
	\p(dx, dw) = \delta_{F (w)} (dx) \mathds{W}(dw).
	\]
Let \(((\Omega, \mathcal{F}, (\mathcal{F}_t)_{t \geq 0}, \P), W)\) be a driving system which supports two solution processes \(X\) and \(Y\). Recalling that joint weak uniqueness holds, we obtain
\[
\P\big(X = F(\overline{W})\big) = \P\big(Y = F(\overline{W})\big) = \iint \1 \{x = F (w)\} \p (dx, dw) = 1.
\]
Consequently, strong uniqueness holds and the proof is complete.
\end{proof}

Let us relate weak solutions to so-called \emph{mild solutions}, which are also frequently used in the literature (see, e.g. \cite{ondrejat-diss,tappe20}). Let \(L_2\) be the space of radonifying operators \(H \to \H\).
The following proposition is a direct consequence of \cite[Theorem~13]{ondrejat-diss}.
\begin{proposition}\label{prop: mild sol}
Assume that \(\H\) is 2-smooth and that \(A\) is the generator of a \(C_0\)-semigroup \((S_t)_{t \geq 0}\) on \(\h\).
	Let \((\B, W)\) be a driving system which supports a continuous \(\H\)-valued adapted process \(X\)
such that a.s.
\[
\int_0^t \|S_{t - s} \u_s (X)\|^2_{L_2} ds < \infty, \quad t \in \mathbb{R}_+.
\]
Then, \(X\) is a solution process on \((\B, W)\) if and only if \eqref{eq: sol inte} holds and a.s.
	\[
	X_t = S_t X_0 + \int_0^t S_{t - s} \v_s (X) ds + \int_0^t S_{t - s} \u_s(X) d W_s, \quad t \in \mathbb{R}_+.
	\]
\end{proposition} 
This proposition shows that certain results from the literature are special cases of ours. For instance, \Cref{theo: main1} generalizes \cite[Theorem 1.3]{tappe20}, and \Cref{coro: DYW} generalizes \cite[Theorem 1.6]{tappe20}.

We end this section with a comment on a possible application of our results. It is interesting to prove strong uniqueness for SPDEs. Similar to the finite dimensional case, \Cref{coro: DYW} shows that strong uniqueness can be deduced from weak uniqueness and strong existence. This strategy is e.g. interesting for equations of the type 
\begin{align} \label{eq: OU with drift}
d X_t = (A X_t + \v_t (X)) dt + d W_t,
\end{align}
whose weak properties can be deduced via Girsanov's theorem from the corresponding properties of the Ornstein--Uhlenbeck equation
\[
d X_t = A X_t dt + d W_t,
\]
see \cite[Appendix I]{roeckner15} for such an argument. In other words, by the Yamada--Watanabe theorems (\Cref{coro: DYW} and \cite[Theorem 5.3]{kunze13}), typically strong existence and uniqueness are equivalent for \eqref{eq: OU with drift}.
More generally, Girsanov's theorem can be used to deduce weak properties for equations of the type
\[
d X_t = (A X_t + \u_t (X) \v_t (X)) dt + \u_t (X) d W_t
\]
from the corresponding properties of the equation
\[
d X_t = A X_t dt + \u_t (X) d W_t.
\]
It is interesting to note that the strong uniqueness properties of \eqref{eq: OU with drift} turn out to be quite subtle for general non-anticipating \(\v\), in fact more subtle than for Markovian \(\v\). 
For suitable linearities \(A\), it was proven in \cite{DFPR13,DFPR14} for a Hilbert space setting that the Markovian equation
\[
d X_t = (AX_t + \v(X_t))dt + d W_t
\]
satisfies strong existence for every (locally) bounded \(\v\). This remarkable result is not true for non-anticipating \(\v\). Indeed, Tsirel'son's example (\cite[Section~V.18]{RW2}) shows that even for bounded non-anticipating \(\v\) the SPDE \eqref{eq: OU with drift} might not satisfy strong uniqueness.

\section{Proof of Theorem \ref{theo: main1}}\label{sec: pf}
The \emph{if} implication is obvious. Thus, we will only prove the \emph{only if} implication. Assume that weak uniqueness holds for the SPDE \eqref{eq: SPDE}.

		Let \(X\) be a solution process to the SPDE \eqref{eq: SPDE} which is defined on a driving system \(((\Omega^*, \mathcal{F}^*, (\mathcal{F}^*_t)_{t \geq 0}, \P^*), W)\).
		We take a second driving system \(((\Omega^o, \mathcal{F}^o, (\mathcal{F}^o_t)_{t \geq 0}, \P^o), B)\) and set
		\[
		\Omega \triangleq \Omega^* \times \Omega^o, \qquad \mathcal{F} \triangleq \mathcal{F}^* \otimes \mathcal{F}^o, \qquad \P \triangleq \P^* \otimes \P^o.
		\]
		Define \(\mathcal{F}_t\) to be the \(\P\)-completion of the \(\sigma\)-field \(\bigcap_{s > t} (\mathcal{F}^*_s \otimes \mathcal{F}^o_s)\). 
		In the following the filtered probability space \(\mathbb{B} = (\Omega, \mathcal{F}^\P, (\mathcal{F}_t)_{t \geq 0}, \P)\) will be our underlying space. We extend \(X, W\) and \(B\) to \(\B\) by setting
		\[
		X (\omega^*, \omega^o) \equiv X(\omega^*), \qquad W (\omega^*, \omega^o) \equiv W (\omega^*), \qquad B(\omega^*, \omega^o) \equiv B(\omega^o)
		\]
		for \((\omega^*, \omega^o) \in \Omega\).
		It is easy to see that \((\B, W)\) and \((\B, B)\) are again driving systems and that \(X\) is a solution process on \((\B, W)\).
		
		For a closed linear subspace \(H^o\) of \(H\) we denote by \(\operatorname{pr}_{H^o}\) the orthogonal projection onto \(H^o\). For \((\omega, t) \in \C \times \mathbb{R}_+\) we define 
		\[
		\phi_t(\omega) \triangleq \on{pr}_{\on{ker} (\u_t (\omega))} \in L (H), \qquad \psi_t (\omega) \triangleq \on{Id}_{H} - \ \phi_t (\omega) \in L(H).
		\]
		Let us summarize some basic properties of \(\phi\) and \(\psi\):
		\begin{equation}\label{eq: prop phi psi}
		\begin{split}
		\phi^2 =  \phi, \quad \psi^2 = \psi, \quad  \u \phi = 0_\H,\quad \u \psi =  \u, \quad \phi \psi =  \psi  \phi = 0_H.
		\end{split}
		\end{equation}
		The following lemma follows from an approximation argument (see the proof of \cite[Lemma 9.2]{ondrejat-diss} for details).
		\begin{lemma}\label{lem:  inte}
		The processes \(\phi = (\phi_t)_{t \geq 0}\) and \(\psi = (\psi_t)_{t \geq 0}\) are progressively measurable as processes on the canonical space \((\C, \mathcal{C}, \bC)\). 
		\end{lemma}
		
		By \Cref{lem:  inte}, we can define a sequence \(V = (V^k)_{k \in \mathbb{N}}\) of continuous local martingales via
		\[
		V^k \triangleq \int_0^\cdot \langle \phi_t (X) e_k, d W_t \rangle_H + \int_0^\cdot \langle \psi_t (X)e_k, d B_t \rangle_H, \quad k \in \mathbb{N}.
		\]
		The following lemma is the technical core of the proof. We postpone its proof till the proof of \Cref{theo: main1} is complete.
		\begin{lemma}\label{lem: main1}
			The process \(V\) is a standard \(\mathbb{R}^\infty\)-Brownian motion. Moreover, \(V\) is independent of \(X\), i.e. the \(\sigma\)-fields \(\sigma (\overline{V}_t, t \in \mathbb{R}_+)\) and \(\sigma (X_t, t \in \mathbb{R}_+)\) are independent, where \(\overline{V}\) is defined by the formula
			\[
			\overline{V} \triangleq \sum_{k = 1}^\infty V^k J e_k.
			\]
		\end{lemma}
		For every \(k \in \mathbb{N}\), \Cref{prop: simple chain rule} in \Cref{app: TT} and \eqref{eq: prop phi psi} yield that
		\begin{align*}
		    \int_0^\cdot \langle \phi_t (X) e_k, d V_t \rangle_{H} &= \int_0^\cdot \langle \phi_t (X) \phi_t (X) e_k, d W_t \rangle_{H} + \int_0^\cdot \langle \psi_t (X) \phi_t (X) e_k, d B_t \rangle_{H}
		    \\&= \int_0^\cdot \langle \phi_t (X) e_k, d W_t \rangle_{H},
		\end{align*}
and consequently, 
\[
\beta^k = \int_0^\cdot \langle \psi_t (X) e_k, d W_t\rangle_H + \int_0^\cdot \langle \phi_t (X) e_k, d V_t \rangle_H, \quad k \in \mathbb{N}.
\]
	By the construction of the stochastic integral, the law of the second term is determined by the law of \((X, V)\), cf. \cite[Proposition 17.26]{Kallenberg} for a similar argument in a finite dimensional setting. In the following we explain that the same is true for the first term, borrowing some ideas from the proof of \cite[Lemma 9.2]{ondrejat-diss}. In fact, we even show that its law is determined by the law of \(X\).
	Define
	\[
	H (t, \omega, x, y^*) \triangleq \| \u_t(\omega)^* y^* - \psi_t (\omega) x \|_{H}, \quad (t, \omega, x, y^*) \in \mathbb{R}_+ \times \mathbb{C} \times H \times \h^*.
	\]
	\begin{lemma}\label{lem: H}
	   For every \(T > 0\) and \(x \in H\) there exists a sequence \((\s^m)_{m \in \mathbb{N}}\) of progressively measurable \(\h^*\)-valued processes on \((\mathbb{C}, \mathcal{C}, \bC)\) such that 
	   \[
	   H(t, \omega, x, \s^m_t(\omega)) \leq \tfrac{1}{m}, \quad (t, \omega, m) \in [0, T] \times \mathbb{C} \times \mathbb{N}.
	   \]
	\end{lemma}
	\begin{proof}
	   We verify the prerequisites of \cite[Proposition 8.8]{ondrejat-diss}: The process \(H(\cdot, \cdot, x , y^*)\) is progressively measurable by \Cref{lem:  inte}. It is clear that \(y^* \mapsto H (t, \omega, x, y^*)\) is continuous (when \(\h^*\) is equipped with the (operator) norm topology). Finally, we show that \(\{y^* \in \h^* \colon H(t, \omega, x, y^*) < 1/m\} \not = \emptyset\) for every \(m \in \mathbb{N}\). Fix \((t, \omega) \in \mathbb{R}_+ \times \mathbb{C}\) and note that
	   \[
	   \psi_t (\omega) (H) \subseteq \on{ker}(\u_t (\omega) )^\perp = \overline{\u_t(\omega) ^* (\h^*)} \subseteq H,\]  
	   cf. \cite[Satz III.4.5]{werner11}. Thus, there exists a sequence \((y^*_m)_{m \in \mathbb{N}} \subset \h^*\) such that 
	   \[
	   \lim_{m \to \infty}\|\u_t (\omega)^* y^*_m - \psi_t (\omega) x\|_H = 0.
	   \]
	   We conclude that \(\{y^* \in \h^* \colon H(t, \omega, x, y^*) < 1/m\} \not = \emptyset\) for every \(m \in \mathbb{N}\). 
	   In summary, the claim follows from \cite[Proposition 8.8]{ondrejat-diss}.
	\end{proof}
	
	Fix \(T > 0\) and \(x \in H\) and let \((\s^m)_{m \in \mathbb{N}}\) be as in \Cref{lem: H}. Then, \Cref{prop: conv ondre} in \Cref{app: TT} yields that
	\begin{align*}
	\sup_{t \in [0, T]} \Big| \int_0^t \langle \psi_s (X) x, d W_s \rangle_{H} &- \int_0^t  \langle \sigma_s (X)^* \s^m_s (X), d W_s \rangle_{H} \Big| \to 0
	\end{align*}
	in probability as \(m \to \infty\). 
	Define \(Z = \{Z (y^*) \colon y^* \in \H^*\}\) by
	\[
	Z (y^*) \triangleq \int_0^\cdot \langle \u_t (X)^* y^*, d W_t\rangle_{H}, \quad y^* \in \H^*.
	\]
	Since
	\[
	\int_0^\cdot \langle \sigma_s (X)^* \s^m_s (X), d W_s\rangle_H = \int_0^\cdot \langle d Z_s, \s^m_s (X) \rangle_U
	\]
	by \Cref{prop: simple chain rule} in \Cref{app: TT}, the construction of the stochastic integral implies that
	the law of \(\int_0^\cdot \langle \sigma_s (X)^* \s^m_s (X), d W_s\rangle_{H}\) is determined by the finite dimensional distribution of \((X, Z)\). Thus, also the law of \(\int_0^\cdot \langle \psi_t (X) e_k, d W_s\rangle_{H}\) is determined by the finite dimensional distributions of \((X, Z)\). 
	
		\begin{lemma}\label{lem: U X-mb} 
		For every (finite) random time \(T\colon \C \to \mathbb{R}_+\) and \(y^* \in D(A^*)\) there exists a measurable map \(F \colon \C \to \overline{\mathbb{R}}\) such that a.s. \(Z_{T (X)} (y^*) = F(X)\).
	\end{lemma}
	\begin{proof}
	We define 
	\begin{align*}
	F(\omega) \triangleq \langle \omega (T(\omega)), y^*\rangle_{\H} - \langle x_0, y^*\rangle_{\H} &- \int_0^{T(\omega)} \langle \omega (s), A^* y^* \rangle_{\H} ds - \int_0^{T(\omega)} \langle \v_s (\omega), y^* \rangle_{\H} ds, 
	\end{align*}
set to be \(+ \infty\) if the last integral diverges. The claim now follows from the definition of a weak solution.
	\end{proof}
	\Cref{lem: U X-mb} shows that the finite dimensional distributions of \(\{Z (y^*) \colon y^* \in D(A^*)\}\) are determined by the law of \(X\). We now adapt an argument from the proof of \cite[Lemma 4.1]{kunze13} to extend this observation to \(\{Z (y^*) \colon y^* \in U^*\}\). Define the localizing sequence
	\[
	T_m \triangleq \inf \Big( t \in \mathbb{R}_+ \colon \int_0^t \|\u_s (X) \|^2_L ds \geq m\Big), \quad m \in \mathbb{N}.
	\]
	Recall that \(D(A^*)\) is assumed to be sequentially weak\(^*\) dense. Thus, for every \(y^* \in \H^*\) there exists a sequence \((y_k^*)_{k \in \mathbb{N}} \subset D(A^*)\) such that \(y^*_k \to y^*\) in the weak\(^*\) topology. 
	Fix \(T > 0\) and \(m > 0\) and denote the Lebesgue measure on \([0, T]\) by \(\llambda\). As \((y^*_k)_{k \in \mathbb{N}}\) is bounded by the uniform boundedness principle, the dominated convergence theorem yields that 
	\begin{align*}
	\lim_{k \to \infty} \E \Big[ \int_0^{T \wedge T_m} \langle h (s), \u_s (X)^* y^*_k \rangle_H ds \Big]
	&= \E \Big[ \int_0^{T \wedge T_m} \langle h (s), \u_s (X)^* y^* \rangle_H ds \Big]
	\end{align*}
	for every \(h \in L^2 (\P \otimes \llambda, H)\).
	This means that
	\begin{align*}
	\u (X)^* y^*_k \1_{[0,  T_m]} \to \u (X)^* y^* \1_{[0, T_m]}
	\end{align*} 
	weakly in \(L^2 (\P \otimes \llambda, H)\) as \(k \to \infty\). By Mazur's lemma (\cite[Korollar~III.3.9]{werner11}), there exists a sequence \((x^*_k)_{k \in \mathbb{N}}\) in the convex hull of \((y^*_k)_{k \in \mathbb{N}}\) (and thus in \(D(A^*)\)) such that 
		\begin{align*}
	\u (X)^* x^*_k \1_{[0, T_m]} \to \u (X)^* y^* \1_{[0, T_m]} 
	\end{align*} 
	strongly in \(L^2(\P \otimes \llambda, H)\) as \(k \to \infty\). 
	Hence, \Cref{prop: conv ondre} in \Cref{app: TT} yields that
	\[
	\sup_{s \in [0, T]} \big| Z_{s \wedge T_m} (x^*_k) - Z_{s \wedge T_m} (y^*) \big| \to 0
	\]
	in probability as \(k \to \infty\).
	Finally, we conclude from \Cref{lem: U X-mb} that the finite dimensional distributions of \((X, Z)\) are determined by the law of \(X\).
	
In summary, the law of \((X, W)\) is determined by the law of \((X, V)\) and hence, by \Cref{lem: main1}, it is determined by the law of \(X\). The proof is complete.
\qed
\\

It remains to prove \Cref{lem: main1}:
\\

\noindent
\emph{Proof of \Cref{lem: main1}:}
\emph{Step 1.} 
Recall that each \(V^k\) is a continuous local martingale by the definition of the stochastic integral.
		Denote the quadratic variation process by \([\cdot, \cdot]\).
For \(i, j \in\mathbb{N}\) and \(t \in \mathbb{R}_+\), using \Cref{prop: QV indep BM} in \Cref{app: TT} and the self-adjointness of \(\phi\) and \(\psi\), we obtain
\begin{align*}
    [V^i, V^j]_t &= \Big[ \int_0^\cdot \langle \phi_s (X) e_i, d W_s \rangle_H , \int_0^\cdot \langle \phi_s (X) e_j, d W_s \rangle_H\Big]_t \\&\qquad\qquad+ \Big[ \int_0^\cdot \langle \psi_s (X) e_i, d B_s \rangle_H, \int_0^\cdot \langle \psi_s (X) e_j, d B_s \rangle_H \Big]_t
    \\&= \int_0^t \langle \phi_s (X) e_i, \phi_s (X) e_j\rangle_H ds + \int_0^t \langle \psi_s (X) e_i, \psi_s (X) e_j\rangle_H ds
    \\&= \int_0^t \langle (\phi_s (X) + \psi_s (X) ) e_i, e_j \rangle_H ds = t \1_{\{i = j\}}.
\end{align*}
L\'evy's characterization implies that \(V\) is a standard \(\mathbb{R}^\infty\)-Brownian motion.\\\quad\\
\noindent
\emph{Step 2.} In this step we prepare the proof of the independence of \(V\), or more precisely \(\overline{V}\), and \(X\). Let \(C^2_b (\mathbb{R})\) be the set of bounded twice continuously differentiable functions with bounded first and second derivative.
\begin{lemma}\label{lem: main3}
	Let \(\overline{Y}\) be a continuous adapted \(\overline{H}\)-valued process starting at \(\overline{Y}_0 = 0_{\overline{H}}\). For \(h \in \overline{H}\) set \(\overline{Y} (h) \triangleq \langle \overline{Y}, h\rangle_{\overline{H}}\).
	The following are equivalent: 
	\begin{enumerate} \item[\textup{(i)}] \(\overline{Y}\) is a trace class Brownian motion with covariance \(J J^*\).
		\item[\textup{(ii)}] For all \(f \in C^2_b(\mathbb{R})\) with \(\inf_{x \in \mathbb{R}} f(x) > 0\) and \(f (0) = 1\), and all \(h \in \overline{H}\) the process
		\begin{align}\label{eq: Mf}
		M^f \triangleq f(\overline{Y}(h) )\exp \Big( - \frac{\langle J J^* h, h \rangle_{\overline{H}}}{2} \int_0^\cdot \frac{f''(\overline{Y}_s (h)) ds}{f (\overline{Y}_s (h))} \Big)
		\end{align}
		is a martingale.
	\end{enumerate}
\end{lemma}
\begin{proof}
	By the classical martingale problem for (one dimensional) Brownian motion (see, e.g. \cite[Theorem 4.1.1]{SV}) and \cite[Proposition 4.3.3]{EK}, (ii) holds if and only if \(\overline{Y} (h)\) is a one dimensional Brownian motion with covariance \(\langle J J^* h , h \rangle_{\overline{H}}\) for all \(h \in \overline{H}\). This yields the claim.
\end{proof}

For \(f = g(\langle \cdot, y^*\rangle_\H)\) with \(g \in C^2 (\mathbb{R})\) and \(y^* \in D(A^*)\) we define 
\begin{align*}
\mathcal{L} f (\X, t) \triangleq g'(\langle \X_t, y^* \rangle_\H) \big( \langle \X_t&, A^* y^*\rangle_\H +  \langle \v_t (\X), y^*\rangle_\H \big) \\&+ \tfrac{1}{2} g'' (\langle \X_t, y^*\rangle_\H) \langle \u_t (\X)^* y^*,\u_t (\X)^* y^*\rangle_{H}.
\end{align*}
Furthermore, we set 
\[
\mathfrak{X} \triangleq \big\{f = g(\langle \cdot, y^* \rangle_\H) \colon g \in C^2(\mathbb{R}), y^* \in D(A^*) \big\}.
\]
The following is a version of \cite[Theorem 3.6]{kunze13} for our framework.
\begin{lemma}\label{lem: main2}
	A probability measure \(\Q'\) on \((\C, \mathcal{C}, \bC)\) is the law of a solution process to the SPDE \eqref{eq: SPDE} if and only if \(\Q'\)-a.s. \(\X_0 = x_0\) and 
	\[
	\int_0^t \|\v_s(\X)\|_\H ds + \int_0^t \|\u_s(\X)\|_{L}^2 ds < \infty, \quad t \in \mathbb{R}_+,
	\]
	and for all \(f \in \mathfrak{X}\) the process
	\begin{align}\label{eq: Kf}
	K^f \triangleq f(\X) - f(x_0) - \int_0^\cdot \mathcal{L} f(\X, s) ds
	\end{align}
	is a local \((\bC^{\Q'}, \Q')\)-martingale, where \(\bC^{\Q'}\) denotes the \(\Q'\)-completion of \(\bC\). Furthermore, for every solution process \(X\) to the SPDE \eqref{eq: SPDE} the process \(K^f \circ X\) is a local martingale on the corresponding driving system.
\end{lemma}
\begin{proof}
	The structure of the proof is classical and similar to the finite dimensional case (see, e.g. \cite[Chapter 5.4]{KaraShre}). 
Let \(\Q'\) be a solution measure to the SPDE \eqref{eq: SPDE} and let \((\B, W, X)\) be a weak solution such that \(\B=(\Omega, \mathcal{F}, (\mathcal{F}_t)_{t \geq 0}, \P)\) and \(\Q' = \P \circ X^{-1}\). Take \(f = g(\langle \cdot, y^*\rangle_\H) \in \mathfrak{X}\).
Then, It\^o's formula yields that 
\begin{align}\label{eq: ito Kf}
K^f \circ X = \int_0^\cdot g'(\langle X_s, y^* \rangle_\H) d\hspace{0.02cm} \Big(\int_0^s \langle \u_u(X)^* y^*, dW_u\rangle_H\Big).
\end{align}
Hence, \(K^f \circ X\) is a local martingale. Due to \cite[Remark 10.40]{J79}, the local martingale property transfers to the canonical space \((\C, \mathcal{C}^{\Q'}, \bC^{\Q'}, \Q')\) and the \emph{only if} implication follows. 

Conversely, let \(\Q'\) be as in the statement of the lemma. Then, using the hypothesis with \(g(x) = x\) and \(g(x) = x^2\) and similar arguments as in the proof of \cite[Proposition~5.4.6]{KaraShre}, for every \(y^* \in D(A^*)\) it follows that
\begin{align}\label{eq: Y(y)}
\Y (y^*) \triangleq \langle \X, y^*\rangle_\H - \langle \X_0, y^*\rangle_\H - \int_0^\cdot \langle \X_s, A^* y^*\rangle_\H ds -\int_0^\cdot \langle \v_s (\X), y^*\rangle_\H ds
\end{align}
is a local \((\bC^{\Q'}, \Q')\)-martingale with quadratic variation \[[ \Y (y^*), \Y (y^*) ] = \int_0^\cdot \langle \u_s (\X)^* y^*, \u_s(\X)^* y^*\rangle_{H} ds.\]
As \(D(A^*)\) is supposed to be weak\(^*\) dense in \(\H^*\), it separates points of \(\H\). Thus, we deduce from \cite[Theorem 3.1]{ondrejat07} that, possibly on an extension of the filtered probability space \((\C, \mathcal{C}^{\Q'}, \bC^{\Q'}, \Q')\), there exists a standard \(\mathbb{R}^\infty\)-Brownian motion \(W\) such that 
\[
\Y(y^*) = \int_0^\cdot \langle \u_s (\X)^* y^*, d W_s\rangle_H, \quad y^* \in D(A^*).
\]
Due to  \eqref{eq: Y(y)}, we conclude the \emph{if} implication. The proof is complete.
\end{proof}
Define \(M^f\) and \(K^g\) as in \eqref{eq: Mf} and \eqref{eq: Kf} with \(\overline{Y}\) replaced by \(\overline{V}\) and \(\X\) replaced by \(X\). It follows from \Cref{lem: main3} and Step 1 that \(M^f\) is a martingale. Similarly, because \(X\) is a solution process to the SPDE \eqref{eq: SPDE}, \(K^g\) is a local martingale by \Cref{lem: main2}.
We now show that \([M^f, K^g] = 0\).
It\^o's formula yields that
\begin{align}\label{eq: Ito Mf}
dM^f_t = \exp \Big( - \frac{\langle J J^* h, h \rangle_{\overline{H}}}{2} \int_0^t \frac{f'' (\overline{V}_s (h)) ds}{f(\overline{V}_s (h))} \Big) f'(\overline{V}_t (h)) d  \overline{V}_t (h).
\end{align}
Using \Cref{prop: QV indep BM} in \Cref{app: TT}, we deduce from \eqref{eq: prop phi psi} that
\begin{align*}
    \Big[ \overline{V}(h), \int_0^{\cdot} \langle \u_s(X)^* y^*, d W_s\rangle_{H}\Big] &= \Big[ \int_0^\cdot \langle \phi_t (X) J^*h, d W_s \rangle_{H}, \int_0^\cdot \langle \u_s (X)^* y^*, d W_s \rangle_{H} \Big] \\&= \int_0^\cdot \langle \u_t (X) \phi_t (X) J^* h,y^* \rangle_{\H} ds = 0.
\end{align*}
In view of \eqref{eq: ito Kf} and \eqref{eq: Ito Mf}, we conclude that \([M^f, K^g] = 0\). \\\quad\\
\noindent
\emph{Step 3:} We are in the position to follow the proof of \cite[Theorem 2.3]{criens20}. 
More precisely, we deduce the independence of \(V\) and \(X\) from \([M^f, K^g] = 0\). 
For \(n \in \mathbb{N}\) set
\[
T_n \triangleq \inf (t \in \mathbb{R}_+ \colon |K^g_t| \geq n), \qquad K^{g, n} \triangleq K^g_{\cdot \wedge T_n}.
\]
As \(K^g\) has continuous paths, \(K^{g, n}\) is bounded on bounded time intervals and consequently, \(K^{g, n}\) is a martingale. Step 2 yields that \([M^{f}, K^{g, n}] = [M^f, K^g]_{\cdot \wedge T_n} = 0\). Hence, by integration by parts, the process \(M^f K^{g, n}\) is a local martingale and a true martingale, because it is bounded on bounded time intervals.
Next, fix a bounded stopping time \(S\) and
define a measure \(\Q'\) on \((\Omega, \mathcal{F})\) as follows:
\[
\Q' (G) \triangleq \E^\P \big[ M^f_S \1_G \big], \qquad G \in \mathcal{F}.
\]
As \(M^f_0 = 1\), the optional stopping theorem shows that \(\Q'\) is a probability measure.
Since \(M^f, K^{g, n}\) and \(M^f K^{g, n}\) are \(\P\)-martingales, we deduce again from the optional stopping theorem that for every bounded stopping time \(T\) 
\begin{align*}
\E^{\Q'} \big[ K^{g, n}_T \big] &= \E^\P \big[M^f_S K^{g, n}_T\big]
\\&= \E^\P \big[ M^f_S \1_{\{S \leq T\}} \E^\P\big[K^{g, n}_T | \mathcal{F}_{S \wedge T} \big] + K^{g, n}_T \1_{\{T < S\}} \E^\P\big[M^f_S | \mathcal{F}_{S \wedge T} \big]\big]
\\&= \E^\P \big[ M^f_SK^{g, n}_{S \wedge T}\1_{\{S \leq T\}} + K^{g, n}_T M^f_{S \wedge T}\1_{\{T < S\}}\big]
\\&= \E^\P \big[ M^f_{S \wedge T} K^{g, n}_{S \wedge T} \big] =
 0. 
\end{align*}
Thus, because \(T\) was arbitrary and \(T_n \nearrow \infty\) as \(n \to \infty\), \(K^{g}\) is a local \(\Q'\)-martingale. As \(g\) was arbitrary, we deduce from \Cref{lem: main2} and \cite[Remark~10.40]{J79} that \(\Q' \circ X^{-1}\) is a solution measure to the SPDE \eqref{eq: SPDE}. The weak uniqueness assumption now implies that \(\P \circ X^{-1} = \Q' \circ X^{-1}\). 
Next, fix a set \(F \in \sigma (X_t, t \in \mathbb{R}_+)\) such that \(\P (F) > 0\) and set 
\[
\Q^* (G) \triangleq \frac{\P(G, F)}{\P(F)},\quad G \in \mathcal{F}.
\]
Clearly, \(\Q^*\) is a probability measure on \((\Omega, \mathcal{F})\).
Recalling \(\P(F) = \Q'(F)\), we obtain 
\[
\E^{\Q^*} \big[ M^f_S \big] = \frac{\Q'(F)}{\P(F)} =  1.
\]
Thus, because \(S\) was arbitrary, \(M^f\) is a \(\Q^*\)-martingale. 
Since \(f\) was arbitrary, \Cref{lem: main3} yields that \(\overline{V}\) is a trace class \(\Q^*\)-Brownian motion with covariance \(J J^*\). Consequently, for every \(G \in \sigma (\overline{V}_t, t \in \mathbb{R}_+)\) we have
\begin{align*}
\P(G, F) = \Q^*(G) \P(F) = \P(G) \P(F).
\end{align*}
Since this equality holds trivially whenever \(F \in \sigma (X_t, t \in \mathbb{R}_+)\) satisfies \(\P(F) = 0\), we conclude that \(V\) and \(X\) are independent. The proof is complete.
\qed

\appendix
\section{Some Facts for Stochastic Integrals} \label{app: TT}
In the following all processes are defined on a fixed filtered probability space (with complete right-continuous filtration). 
Let \(W\) and \(B\) be two standard \(\mathbb{R}^\infty\)-Brownian motions and let \(\phi = (\phi_t)_{t \geq 0}, \psi = (\psi_t)_{t \geq 0}\) and \(\phi^n = (\phi^n_t)_{t \geq 0}\) be \(H\)-valued progressively measurable processes such that a.s.
\[
\int_0^t \big(\|\phi_s\|^2_H + \|\psi_s\|^2_H + \|\phi^n_s\|^2_H \big) ds < \infty, \quad t \in \mathbb{R}_+.
\]
We start with a basic property of stochastic integrals,
which we use throughout the article without further reference. 
Recall that \([\cdot, \cdot]\) denotes the quadratic variation process.
\begin{proposition}
\begin{align*}
    \Big[ \int_0^\cdot \langle \phi_s, d W_s \rangle_H, \int_0^\cdot \langle \psi_s, d W_s \rangle_H \Big] = \int_0^\cdot \langle \phi_s, \psi_s \rangle_H ds.
\end{align*}
\end{proposition}

The following proposition is a direct consequence of \cite[Proposition 4.1]{ondrejat-diss}.
\begin{proposition} \label{prop: conv ondre}
If for some \(T > 0\)
\[
\int_0^T \|\phi_s - \phi^n_s\|^2_H ds \to 0 \text{ as } n \to \infty
\]
in probability, then 
\[
\sup_{t \in [0, T]} \Big| \int_0^t \langle \phi_s, d W_s \rangle_H - \int_0^t \langle \phi^n_s, d W_s \rangle_H \Big| \to 0 \text{ as } n \to \infty
\]
in probability as well.
\end{proposition}

The next proposition follows from \cite[Proposition 4.7]{rem20}.
\begin{proposition} \label{prop: QV indep BM}
If \(W\) and \(B\) are independent, then 
\[
\Big[ \int_0^\cdot \langle \phi_s, d W_s \rangle_H, \int_0^\cdot \langle \psi_s, d B_s \rangle_H \Big] = 0.
\]
\end{proposition}
 
 Finally, we also provide a simple chain rule, which can be proven by first checking it for simple processes and then using an approximation argument. 
\begin{proposition} \label{prop: simple chain rule}
Let \(g = (g_t)_{t \geq 0}\) be an \(L\)-valued progressively measurable process and let \(\theta^* = (\theta^*_t)_{t \geq 0}\) be a \(U^*\)-valued progressively measurable process such that a.s.
\[
\int_0^t \big(\|g_s\|^2_L + \|g^*_s \theta_s^*\|^2_H \big) ds < \infty,\quad t \in \mathbb{R}_+, 
\]
and define a cylindrical local martingale \(Z = \{Z (y^*) \colon y^* \in U^*\}\) by 
\[
Z (y^*) \triangleq \int_0^\cdot \langle g^*_s y^*, d W_s \rangle_H, \quad y^* \in U^*.
\]
Then, the stochastic integrals
\[
\int_0^\cdot \langle g^*_s\theta_s^*, d W_s \rangle_H, \qquad \int_0^\cdot \langle d Z_s, \theta_s^* \rangle_U 
\]
are well-defined (where the second integral is defined as in \textup{\cite{miro98}}) and 
\[
\int_0^\cdot \langle g^*_s\theta_s^*, d W_s \rangle_H= \int_0^\cdot \langle d Z_s, \theta_s^* \rangle_U.
\]
\end{proposition}

\bibliographystyle{plain}

\end{document}